\documentclass{amsart}
\usepackage{amsfonts}
\usepackage{amsmath,amssymb}
\usepackage{amsthm}
\usepackage{amscd}
\usepackage{graphics}
\usepackage{graphicx}
\theoremstyle{definition}{
\newtheorem{Def}{{\rm Definition}}
\newtheorem{Ex}{{\rm Example}}
\newtheorem{Rem}{{\rm Remark}}

}
\newtheorem{Cor}{Corollary}
\newtheorem{Prop}{Proposition}
\newtheorem{Thm}{Theorem}

\begin{document}
\title[Explicit fold maps whose singular value sets may have double points]{Surgery operations to fold maps to construct fold maps whose restrictions to the singular sets may not be embeddings}
\author{Naoki Kitazawa}
\address{Institute of Mathematics for Industry, Kyushu University, 744 Motooka, Nishi-ku Fukuoka 819-0395, Japan
 TEL (Office): +81-92-802-4402 FAX (Office): +81-92-802-4405}
\email{n-kitazawa@imi.kyushu-u.ac.jp}
\begin{abstract}
Constructing {\it Morse} functions and their higher dimensional versions or {\it fold} maps is
 fundamental, important and challenging in investigating the topologies and the differentiable structures of differentiable manifolds
 via Morse functions, fold maps and more general generic maps. It is one of important and interesting branches of the singularity theory of differentiable maps and applications to geometry of manifolds.

In this paper we present fold maps with information of cohomology rings of their {\it Reeb spaces}. Reeb spaces are defined as the spaces of all connected components of all preimages, and in suitable situations inherit topological information such as homology groups and cohomology rings of the manifolds. Previously, the author demonstrated construction of fold maps in various cases : key methods are surgery operations to manifolds and maps and in this paper, we present more useful surgery operations and by them we construct new fold maps. More precisely, fold maps such that the restrictions to the singular sets may not be embeddings: the {\it singular set} of a smooth map is set of all singular points and note that for a fold map, the singular set is a closed submanifold with no boundary and the restriction to this set of the original map is a smooth immersions of codimension $1$.
\end{abstract}
\maketitle
\section{Introduction and fundamental notation and terminologies.}
\label{sec:1}
As the well-known classical theory of Morse functions shows, investigating the topologies and the differentiable structures of differentiable manifolds
 via {\it Morse} functions, {\it fold} maps and more general generic maps is one of important and interesting branches of the singularity theory of differentiable maps and applications to geometry of manifolds. The present paper is a study on this. More precisely, this paper is on explicit construction on fold maps and homology groups and cohomology rings of {\it Reeb spaces} of these maps. {\it Reeb spaces} are defined as the spaces of all connected components of all preimages, and in suitable situations inherit homology groups, cohomology rings, and so on, of the manifolds, which we will explain later again. Thus investigating Reeb spaces connects to understanding algebraic topological and differential topological properties of the world of manifolds.  

\subsection{Fold maps.}

{\it Fold} maps are smooth maps regarded as higher dimensional versions of Morse functions.

For the explanation of the strict definition and fundamental properties of a {\it fold} map, we explain fundamental terminologies and notation related to {\it singular points} of smooth ($C^{\infty}$) maps. They are used throughout the present paper. Throughout the paper, manifolds are assumed to be smooth (of class $C^{\infty}$) and maps between (smooth) manifolds are assumed to be smooth (of class $C^{\infty}$) unless otherwise stated: for example, in the presentation of a sketch of the proof of Proposition \ref{prop:4} and so on, {\it piecewise smooth} and PL manifolds which may not be smooth appear.

A {\it singular point} of a smooth map $c$ is a point at which the rank of the differential of the map is smaller than the dimension of the target space. A {\it singular value} of the map is a point which is a value at some singular point. The set $S(c)$ of all singular points is the {\it singular set} of the map. The {\it singular value set} is the image $c(S(c))$. The {\it regular value set} of $c$ is the complement of $c(S(c))$ in the target space. A {\it singular} {\rm (}{\it regular}{\rm )} {\it value} of the map is a point in the singular (resp. regular) value set.

A diffeomorphism is always assumed to be smooth and the {\it diffeomorphism group} of a manifold is the group of all diffeomorphisms on the manifold.

\begin{Def}
\label{def:1}
Let $m>n \geq 1$ be integers. A smooth map from an $m$-dimensional smooth manifold with no boundary into an $n$-dimensional smooth manifold with no boundary is said to be a {\it fold} map if at each singular point $p$, the map is represented as
$$(x_1, \cdots, x_m) \mapsto (x_1,\cdots,x_{n-1},\sum_{k=n}^{m-i}{x_k}^2-\sum_{k=m-i+1}^{m}{x_k}^2)$$
 for some suitable coordinates and an integer $0 \leq i(p) \leq \frac{m-n+1}{2}$.
\end{Def}

For a fold map, the following properties are fundamental.
\begin{itemize}
\item For any singular point $p$, the $i(p)$ in Definition \ref{def:1} is unique. 
\item The set consisting of all singular points of a fixed index of the map is a smooth and closed submanifold of the source manifold with no boundary of dimension $n-1$. 
\item The restriction to the singular set of the original map is a smooth immersion.
\end{itemize}
Here, we define the {\it index} of $p$ by $i(p)$.

We will define a {\it crossing} and a {\it normal} crossing of a family of smooth immersions.

Let $a$ be a positive integer and $\{c_j:X_j \rightarrow Y\}_{j=1}^{a}$ be a family of $a$ smooth immersions $c_j$ from $m_j$-dimensional smooth manifolds $X_j$ without boundaries into an $n$-dimensional smooth manifold $Y$ with no boundary.
A {\it crossing} of the family of these smooth immersions is a point $y \in Y$ such that 
${\bigcup}_{j=1}^a {c_j}^{-1}(y)$ has at least two points. A crossing is said to be {\it normal} if the following properties hold.

For a smooth manifold $X$, we denote the tangent bundle by $TX$ and the tangent vector at $p \in X$ by $T_p X \subset TX$.

\begin{enumerate}
\item The disjoint union ${\bigcup}_{j=1}^a {c_j}^{-1}(y)$ of these preimages is a finite set consisting of exactly $b>1$ points.
\item Let $\{p_j\}_{j=1}^{b}$ be the set just before. Let $n_y(j)$ be the number satisfying $p_j \in X_{n_y(j)}$: we can define uniquely. If we consider the intersection ${\bigcap}_{j=1}^{b} {d {c_{n_y(j)}}}_{p_j}(T_{p_j}X_{n_y(j)})$ of the images of the differentials at all points and we denote the dimension of this by $i_y(d)$, then the relation $i_y(d)+{\Sigma}_{j=1}^b (n-m_{n_y(j)})=n$ holds.
\end{enumerate}

We can define a {\it normal crossing} of a single immersion (the case $a=1$).

In the present paper, we only consider crossings such that preimages consist of exactly two points.

A {\it stable} fold map is a fold map 
whose restriction to the singular set is a smooth immersion such that the crossings of the restriction to the singular set of the original fold map are always normal.
 For systematic explanations on {\it stable} (fold) maps, see
 \cite{golubitskyguillemin} for example.
In this paper, we construct stable fold maps such that the restriction to the singular set of the original fold map may have normal crossings.

\subsection{Reeb spaces} 
Reeb spaces are also fundamental and important tools in investigating the topologies of the source manifolds of smooth maps whose codimensions are negative. 

 Let $X$ and $Y$ be topological spaces. For $p_1, p_2 \in X$ and for a continuous map $c:X \rightarrow Y$, 
 we define as $p_1 {\sim}_c p_2$ if and only if $p_1$ and $p_2$ are in
 a same connected component of $c^{-1}(p)$ for some $p \in Y$. 
Thus ${\sim}_{c}$ is an equivalence relation on $X$ and we denote the quotient space $X/{\sim}_c$ by $W_c$.
\begin{Def}
\label{def:2}
We call $W_c$ the {\it Reeb space} of $c$.
\end{Def}

 We denote the induced quotient map from $X$ into $W_c$ by $q_c$. We can define $\bar{c}: W_c \rightarrow Y$ uniquely
 so that the relation $c=\bar{c} \circ q_c$ holds.

\begin{Prop}[\cite{shiota}]
\label{prop:1}
For stable fold maps, the Reeb spaces are polyhedra and the dimensions are equal to the dimensions of the target manifolds.
\end{Prop}

For Reeb spaces, see also \cite{reeb} for example.

\subsection{Reeb spaces of fold maps}

We introduce a class of stable fold maps such that the Reeb spaces inherit much information of algebraic topological invariants of the manifolds in Proposition \ref{prop:2}.

A {\it simple} fold map $f$ is a fold map such that the restriction $q_f {\mid}_{S(f)}$ is injective.

{\it PID} means a principal ideal domain throughout the present paper. 

\begin{Prop}[\cite{saekisuzuoka}  (\cite{kitazawa2} and \cite{kitazawa3})]
\label{prop:2}
Let $m$ and $n$ be integers satisfying $m>n \geq 1$. Let $A$ be a commutative group. Let $M$ be a smooth, closed, connected and orientable manifold of dimension $m$ and $N$ be an $n$-dimensional smooth manifold with no boundary.
  
Then, for a simple fold map $f:M \rightarrow N$ such that preimages of regular values are always disjoint unions of standard spheres and that indices of singular points are always $0$ or $1$, the following three hold.
\begin{enumerate}
\item Three induced homomorphisms
 ${q_f}_{\ast}:{\pi}_j(M) \rightarrow {\pi}_j(W_f)$, ${q_f}_{\ast}:H_j(M;A) \rightarrow H_j(W_f;A)$, and ${q_f}^{\ast}:H^j(W_f;A) \rightarrow H^j(M;A)$ are group isomorphisms and the latter two are isomorphisms of modules over $A$ for $0 \leq j \leq m-n-1$. 
\item Let $A$ be a commutative ring. Let $J$ be the set of all integers greater than or equal to $0$ and smaller than or equal to $m-n-1$ and if ${\oplus}_{j \in J} H^{j}(W_f;A)$ and ${\oplus}_{j \in J} H^{j}(M;A)$ are defined as algebras where the sums and the products are canonically induced from the cohomology rings $H^{\ast}(W_f;A)$ and $H^{\ast}(M;A)$ respectively and where the maximal degrees are $m-n-1$, then $q_f$ induces an isomorphism between the algebras
${\oplus}_{j \in J} H^j(W_f;A)$ and ${\oplus}_{j \in J} H^{j}(M;A)$ and this is a restriction of ${q_f}^{\ast}$ to ${\oplus}_{j \in J} H^j(W_f;A)$.
\item Let $A$ be a PID and let $m=2n$ be hold. In this situation, the rank of the module $H_n(M;A)$ is twice the rank of the module $H_n(W_f;A)$ and in addition if $H_{n-1}(W_f;A)$, which is isomorphic to $H_{n-1}(M;A)$, is a free module over $A$, then the first two modules over $A$ are also free.
\end{enumerate} 
\end{Prop} 
\begin{Rem}
Proposition \ref{prop:2} holds in cases where preimages of regular values may contain homotopy spheres obtained by gluing two copies of a standard closed disc by a diffeomorphism between their boundaries as connected components. The class of such homotopy spheres accounts for the class of all homotopy spheres except $4$-dimensional homotopy spheres which are not standard spheres (such manifolds are still undiscovered). However, we only handle cases where the preimages are disjoint unions of standard spheres.
\end{Rem}
\subsection{Explicit fold maps and their Reeb spaces}
It is fundamental and important to construct explicit fold maps in applying geometric theory of Morse functions and fold maps to understanding of geometric properties of manifolds. However, even on (families of) manifolds which are not so complicated, it is difficult. We present known examples here.

For a topological space $X$, an {\it X-bundle} is a bundle whose fiber is $X$. For a smooth manifold $X$, a {\it smooth} $X$-bundle is an $X$-bundle whose structure group is (a subgroup of) the diffeomorphism group.
\begin{Ex}
\label{ex:1}
\begin{enumerate}
\item
\label{ex:1.1}
The class of {\it special generic} maps is a proper subclass of that of simple fold maps in Proposition \ref{prop:2}.
A {\it special generic} map is a fold map such that the index of each singular point is $0$.
Standard spheres admit special generic maps into arbitrary Euclidean spaces whose dimensions are smaller than or equal to the dimensions of the spheres: canonical projections of unit spheres are simplest ones. In \cite{saeki}, \cite{saeki2}, \cite{saekisakuma}, \cite{wrazidlo} and so on, homotopy spheres which are not diffeomorphic to standard spheres do not admit special generic maps into sufficiently high dimensional Euclidean spaces under the constraint that the dimensions of the Euclidean spaces are assumed to be smaller than those of the homotopy spheres. Moreover, as another fundamental and important property of a special generic map, the maximal degree $j=m-n-1$ in Proposition \ref{prop:2} can be replaced by $j=m-n$.

Last, the Reeb space of a (stable) special generic map $f$ from a closed and connected manifold of dimension $m$
 into ${\mathbb{R}}^n$ satisfying the relation $m>n \geq 1$ is homeomorphic to and regarded as an $n$-dimensional compact and connected manifold we can immerse into ${\mathbb{R}}^n$. The image is regarded as the image of a suitable immersion of the $n$-dimensional manifold. The boundary of the Reeb space and the image $q_f(S(f))$ of the singular set agree. 

Conversely, for integers $m>n \geq 1$ and an $n$-dimensional compact manifold we can immerse
 into ${\mathbb{R}}^n$, we can construct a (stable) special generic map from a suitable closed and connected manifold of dimension $m$ into ${\mathbb{R}}^n$ whose Reeb space is diffeomorphic to the $n$-dimensional manifold.

Moreover, we have the following smooth or linear bundles for a general special generic map $f$ from a closed and connected manifold of dimension $m$ into ${\mathbb{R}}^n$. 
\begin{enumerate}
\item If we restrict the map $q_f$ to the preimage of the interior of the Reeb space, then it gives a smooth $S^{m-n}$-bundle over the interior of the Reeb space.
\item If we restrict the map $q_f$ to the preimage of a small collar neighborhood of the boundary of the Reeb space and consider the composition of this with a canonical projection onto the boundary, then it gives a linear $S^{m-n+1}$-bundle over the boundary.
\end{enumerate}
For integers $m>n \geq 1$ and an $n$-dimensional compact manifold we can immerse into ${\mathbb{R}}^n$, we can construct a (stable) special generic map from a suitable closed and connected manifold of dimension $m$
 into ${\mathbb{R}}^n$ whose Reeb space is diffeomorphic to the $n$-dimensional manifold and these bundles are trivial.

These facts can be seen in \cite{saeki}, in the articles \cite{kitazawa5} and \cite{kitazawa6} by the author, and so on.
\item
\label{ex:1.2}
(\cite{kitazawa}, \cite{kitazawa2} and \cite{kitazawa4}.)
Let $l>0$ be an integer. Let $m>n \geq 1$ be integers. We can construct a stable fold map on a manifold represented as an $l-1$ connected sum of total spaces of smooth $S^{m-n}$-bundles over $S^n$ into ${\mathbb{R}}^n$ (a standard sphere if $l=1$) satisfying the following properties.
\begin{enumerate}
\item The singular value set is ${\sqcup}_{j=1}^{l} \{||x||=j \mid x \in {\mathbb{R}}^n\}$.
\item Preimages of regular values are disjoint unions of standard spheres.
\item In the target space ${\mathbb{R}}^n$, the number of connected components of an preimage increases as we go in the direction of the origin $0 \in {\mathbb{R}}^n$ of the target Euclidean space from a point in the complement of the image. 
\end{enumerate}
The Reeb space has the simple homotopy type of a bouquet of $l-1$ copies of a sphere of dimension $n$ for $l>1$ and that of a point for $l=1$.
\item
\label{ex:1.3}
 \cite{kitazawa5}, \cite{kitazawa6} and so on, are on construction of stable fold maps such that the restrictions to the singular sets are embeddings satisfying the assumption of Proposition \ref{prop:2}. The maps are obtained by finite iterations of surgery operations ({\it bubbling operations}) starting from fundamental fold maps: this operation is introduced by the author in \cite{kitazawa5} respecting ideas of \cite{kobayashisaeki} for example. More precisely, starting from stable special generic maps mainly, by changing maps and manifolds by bearing new connected components of singular (value) sets one after another, we obtain desired maps.

\end{enumerate}
\end{Ex}
\subsection{Construction of explicit fold maps such that the restrictions to the singular sets may have crossings by new surgery operations (bubbling operations) and the organization of this paper}
In this paper, we present further studies on construction in Example \ref{ex:1} (\ref{ex:1.3}). We introduce and use improved versions of bubbling operations.
The organization of the paper is as the following.
In the next section, we introduce a {\it bubbling operation} first introduced in \cite{kitazawa5} by extending the original definition to construct fold maps such that the restrictions to singular sets may have crossings. Defining this extended operation is also a new work in the present paper. 
The last section is devoted to explanations of simple and important examples and main results. We present construction of new families of explicit fold maps via operations before and investigate the cohomology rings of the Reeb spaces. We also explain properties on the cohomology rings which fold maps obtained previously in \cite{kitazawa5} and \cite{kitazawa6} do not satisfy. Proposition \ref{prop:2} and so on are key tools in knowing the cohomology rings of the manifolds from Reeb spaces in suitable cases and we introduce an extended version of this last as Proposition \ref{prop:5}. The proof is done based on the proofs of the related known results,

\section{Bubbling operations and fold maps such that preimages of regular values are disjoint unions of spheres.}
\label{sec:2}
\subsection{Definitions and fundamental properties of Reeb spaces and bubbling operations.}
Throughout this paper, let $m>n \geq 1$ be integers, $M$ be a smooth, closed and connected manifold of dimension $m$, $N$ be a smooth manifold of dimension $n$ with no boundary, and $f:M \rightarrow N$ be a smooth map unless otherwise stated.

In addition, the structure groups of bundles such that their base spaces and fibers are manifolds are assumed to be
 (subgroups of) diffeomorphism groups except some cases in the presentation of a sketch of the proof of Proposition \ref{prop:4} and so on: the bundles are smooth in a word except several cases. A {\it linear} bundle is a smooth bundle whose fiber is a $k$-dimensional unit disc (standard closed disc of a fixed diameter) or the $k$-dimensional unit sphere in ${\mathbb{R}}^{k+1}$ and whose structure groups are subgroups of the $k$-dimensional orthogonal group $O(k)$ and the ($k+1$)-dimensional one $O(k+1)$ acting canonically and linearly, respectively. 
 
We introduce {\it bubbling operations}, first introduced in \cite{kitazawa5}, referring to the article. We revise some notions and terminologies from the original definition and define an extended version.
\begin{Def}
\label{def:3}
For a stable fold map $f:M \rightarrow N$, let $P$ be a connected
 component of $(W_f-q_f(S(f))) \bigcap {\bar{f}}^{-1}(N-f(S(f)))$, which we may regard as an open manifold diffeomorphic to
 an open manifold $\bar{f}(P)$ in $N$. 
Let $l>0$ and $l^{\prime} \geq 0$ be integers. Assume that there exist families $\{S_j\}_{j=1}^{l}$ and $\{N(S_j)\}_{j=1}^{l}$ of finitely many closed and connected manifolds and total spaces of linear bundles over these manifolds whose fibers are unit discs. 
We also denote by $S_j$ the image of the section obtained by taking the origin for each fiber diffeomorphic to a unit disc for each $N(S_j)$. Assume that the dimensions of $N(S_j)$ are all $n$. Assume also that immersions $c_j:N(S_j) \rightarrow P$ satisfying the following properties exist.
\begin{enumerate}
\item Crossings of the family $\{{c_j} {\mid}_{\partial N(S_j)}:\partial N(S_j) \rightarrow P\}_{j=1}^l$ are normal and ${\bigcup}_{j=1}^l {{c_j} {\mid}_{\partial N(S_j)}}^{-1}(p)$ consists of at most two points for each $p \in P$.
\item Crossings of $\{{c_j} {\mid}_{S_j}:S_j \rightarrow P\}_{j=1}^l$ are normal and ${\bigcup}_{j=1}^l {{c_j} {\mid}_{S_j}}^{-1}(p)$ consists of at most two points for each $p \in P$. 
\item The set of all the crossings of the family $\{{c_j} {\mid}_{S_j}:S_j \rightarrow P\}_{j=1}^l$ of the immersions is a finite set.
\item Let $\{p_{j^{\prime}}\}_{j^{\prime}=1}^{l^{\prime}}$ be the set of all the crossings of the family $\{{c_j} {\mid}_{S_j}:S_j \rightarrow P\}_{j=1}^l$ of the immersions. For each $p_{j^{\prime}}$, there exist one or two integers $1 \leq a(j^{\prime}),b(j^{\prime}) \leq l$ and small standard closed
 discs $D_{2j^{\prime}-1} \subset  S_{a(j^{\prime})}$ and $D_{2j^{\prime}} \subset S_{b(j^{\prime})}$ satisfying the following properties.
\begin{enumerate}
\item (The dimensions) $\dim D_{2j^{\prime}-1}=\dim S_{a(j^{\prime})}$ and $\dim D_{2j^{\prime}}=\dim S_{b(j^{\prime})}$.
\item (The locations of $p_{j^{\prime}}$) $p_{j^{\prime}}$ is in the images of the immersions: $p_{j^{\prime}} \in c_{a(j^{\prime})}({\rm Int }D_{2j^{\prime}-1})$ and $p_{j^{\prime}} \in c_{b(j^{\prime})}({\rm Int }D_{2j^{\prime}})$.
\item If $a(j^{\prime})=b(j^{\prime})$, then $D_{2j^{\prime}-1}$ and $D_{2j^{\prime}}$ do not intersect.
\item If we restrict the bundle $N(S_{a(j^{\prime})})$ over $S_{a(j^{\prime})}$ to $D_{2j^{\prime}-1}$, restrict the bundle $N(S_{b(j^{\prime})})$ over $S_{b(j^{\prime})}$ to $D_{2j^{\prime}}$, and consider the images of the total spaces by $c_{a(j^{\prime})}$ and $c_{b(j^{\prime})}$, which are regarded as manifolds with corners, then they agree as subsets in ${\mathbb{R}}^n$. The restrictions to these spaces are embeddings.
\item The set of all the crossings of the family of $\{{c_j} {\mid}_{S_j}:S_j \rightarrow P\}_{j=1}^l$ is the disjoint union of the $l^{\prime}$ corners of the subsets just before each of which corresponds to $1 \leq j^{\prime} \leq l^{\prime}$. 
\end{enumerate}
\end{enumerate}
We call the family $\{(S_j,N(S_j),c_j:N(S_j) \rightarrow P)\}_{j=1}^{l}$ a {\it normal system of submanifolds} compatible with $f$. 
\end{Def}

\begin{figure}
\includegraphics[width=25mm]{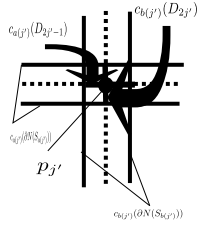}
\caption{Images of immersions around $p_{j^{\prime}}$ with $\dim D_{2j^{\prime}-1}=\dim D_{2j^{\prime}}=1$ (dotted lines are for the images of $S_{a(j^{\prime})}$ and $S_{b(j^{\prime})}$ and thick lines are for the images
 of the boundaries $\partial N(S_{a(j^{\prime})})$ and $N(S_{b(j^{\prime})})$: thick lines are also for the singular value set of a fold map $f^{\prime}$ in Definition \ref{def:5} (or \ref{def:6}) later).}
\label{fig:1}
\end{figure}
In the situation of Definition \ref{def:3}, let $\{N^{\prime}(S_j) \subset N(S_j)\}_{j=1}^{l}$ be the family of total spaces of subbundles of $\{N(S_j)\}_{j=1}^{l}$ over the manifolds whose fibers are standard closed discs and let the diameters of the fibers be all $0<r<1$. For a suitable $r$, same properties as presented in Definition \ref{def:3} hold. In other words, we
 can obtain a family $\{(S_j,N^{\prime}(S_j),{c_j} {\mid}_{N^{\prime}(S_j)}:N^{\prime}(S_j) \rightarrow P)\}_{j=1}^l$ and this is also regarded as a normal system of submanifolds compatible with $f$, by identifying each fiber, which is a standard closed disc of diameter $r$ with a unit disc via the diffeomorphism defined by $t \mapsto \frac{1}{r} t$.

\begin{Def}
\label{def:4}
The family $\{(S_j,N(S_j),c_j:N(S_j) \rightarrow P)\}_{j=1}^{l}$ is said to be a {\it wider normal system supporting} the normal system of submanifolds $\{(S_j,N^{\prime}(S_j),{c_j} {\mid}_{N^{\prime}(S_j)}:N^{\prime}(S_j) \rightarrow P)\}_{j=1}^{l}$ compatible with $f$.
\end{Def}
We have the following immediately.

\begin{Def}
\label{def:5}
For a stable fold map $f:M \rightarrow N$ and an integer $l>0$, let $P$ be a connected
 component of $(W_f-q_f(S(f))) \bigcap {\bar{f}}^{-1}(N-f(S(f)))$ and let $\{(S_j,N(S_j),c_j:N(S_j) \rightarrow P)\}_{j=1}^{l}$ be a normal system of submanifolds compatible with $f$.
Let $\{(S_j,N^{\prime}(S_j),{c_j}^{\prime}:N^{\prime}(S_j) \rightarrow P)\}_{j=1}^{l}$ be a wider normal system supporting this.
Assume that we can construct a stable fold map $f^{\prime}$ on an $m$-dimensional closed manifold $M^{\prime}$ into ${\mathbb{R}}^n$ satisfying the following properties.
\begin{enumerate}
\item $Q$ is the preimage $f^{-1}({\bigcup}_{j=1}^l {c_j}^{\prime}(N^{\prime}(S_j)))$. 
\item $M-{\rm Int} Q$ is realized as a compact submanifold of $M^{\prime}$ of dimension $m$ by considering a suitable smooth embedding $e:M-{\rm Int} Q \rightarrow M^{\prime}$.
\item $f {\mid}_{M-{\rm Int} Q}={f }^{\prime} \circ e {\mid}_{M-{\rm Int} Q}$ holds.
\item ${f}^{\prime}(S({f}^{\prime}))$ is the disjoint union of $f(S(f))$ and ${\bigcup}_{j=1}^n c_j(\partial N(S_j))$.

\end{enumerate}
This enables us to define a procedure of constructing $f^{\prime}$ from $f$. We call it a {\it normal bubbling operation} to $f$. Furthermore, the union ${\bigcup}_{j=1}^l c_j(S_j)$ of the images of the immersions and the family of the images $c_j(S_j)$ for all $c_j$ are called the {\it generating normal systems} of the normal bubbling operation: we call each manifold $S_j$ a {\it generating} manifold of the operation.
\end{Def} 

Ideas for the operations originate from some of \cite{kobayashi}, \cite{kobayashi2} and \cite{kobayashisaeki}. Especially, 
\cite{kobayashi} and \cite{kobayashi2} are on {\it bubbling surgeries} introduced by Kobayashi: a {\it bubbling surgery} is the case where the generating normal system consists of exactly one point.  

In the present paper, we consider normal bubbling operations whose generating manifolds are standard spheres.
Moreover, essentially we consider only stable fold maps such that preimages of regular values are disjoint unions of almost-spheres (standard spheres) and that indices of singular points are $0$ or $1$. Proposition \ref{prop:2} is for simple fold maps satisfying this. Moreover, the codimensions of fold maps are larger than $1$ unless otherwise stated.

More precisely, we consider normal bubbling operations in situations satisfying the following properties.
\begin{Def}
\label{def:6}
For a stable fold map $f:M \rightarrow N$ on an $m$-dimensional closed and connected manifold into an $n$-dimensional manifold with no boundary satisfying $m-n>1$ and an integer $l>0$, let $P$ be a connected
 component of $(W_f-q_f(S(f))) \bigcap {\bar{f}}^{-1}(N-f(S(f)))$ and let $\{(S_j,N(S_j),c_j:N(S_j) \rightarrow P)\}_{j=1}^{l}$ be a normal system of submanifolds compatible with $f$ such that each $S_j$ is a point or a standard sphere. Suppose that we can demonstrate a normal bubbling operation to $f$ to obtain $f^{\prime}$ whose generating normal system is $\{c_j(S_j)\}$ satisfying the following properties.
\begin{enumerate}
\item The union ${\bigcup}_{j=1}^l c_j(N(S_j))$ of the images of the immersions is in an open set $U \subset P$ such that $f {\mid}_{f^{-1}(U)}:f^{-1}(U) \rightarrow U$ is a trivial bundle whose fiber is a standard sphere.
\item The indices of points in the preimage of new connected components in the resulting singular value set are all $1$.
\item For each regular value $p \in U$ of the resulting map, the preimage is a disjoint union of standard spheres.  
\end{enumerate}
We say that this operation is an {\it admissible trivial operation with standard spheres} or that this operation is {\it ATSS}. 
\end{Def}
FIGURE \ref{fig:2} shows for a case (where) some generating manifolds are circles and $n=2$ using a same figure as that of FIGURE \ref{fig:1} or Definition \ref{def:3}. Numbers indicate the numbers of connected components of the preimages of regular values in the corresponding places. 

\begin{figure}
\includegraphics[width=30mm]{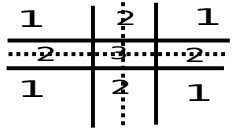}
\caption{The numbers of connected components of the preimages of regular values close to $p_{j^{\prime}}$ in Definition \ref{def:3} or FIGURE \ref{fig:1} for the resulting map $f^{\prime}$.}
\label{fig:2}
\end{figure}

We give an explanation of an explicit local fold map around $p_{j^{\prime}}$.

First in Example \ref{ex:1} (\ref{ex:1.2}), we explain a fold map for $l=1$. We construct a product bundle $D^n \times S^{m-n}$ over $D^n$. We also set a Morse function ${\tilde{f}}_{m-n,[a,b]}$ on $S^{m-n} \times [-1,1]$ onto $[a,b] \subset (0,+\infty) \subset \mathbb{R}$ such that the preimage of the minimum is the boundary, that there exist exactly two singular points, and that singular points are in the interior. We glue the projection of the product bundle and the map
 ${\tilde{f}}_{m-n,[a,b]} \times {\rm id}_{S^{n-1}}:[a,+\infty) \times S^{n-1}$ where we identify the base space $D^n$ of the product bundle with a standard closed disc of dimension $n$ whose diameter is $a$ (note that $S^0$ is a two point set with the discrete topology). 
By gluing suitably, we have a fold map from $S^n \times S^{m-n}$ onto the standard closed disc of dimension $n$ whose center is the origin and diameter is $b$ in ${\mathbb{R}}^n$. This is a desired map. More precisely, see \cite{kitazawa} and \cite{kitazawa4}. 

In this construction, we replace ${\tilde{f}}_{m-n,[a,b]}$ by ${\tilde{f}}_{m-n,[a,b^{\prime})}:={\tilde{f}}_{m-n,[a,b]} {\mid}_{{{{\tilde{f}}_{m-n,[a,b]}}}^{-1}([a,b^{\prime}])}$ where $b^{\prime}<b$ is sufficiently close to $b$. We denote the resulting map onto the standard closed disc of dimension $n$ whose center is the origin and diameter is $b^{\prime}$ in ${\mathbb{R}}^n$ by ${\tilde{f}}_{m,n,b^{\prime}}$. 

We consider the composition of ${\tilde{f}}_{m-n+\dim D_{2j^{\prime}-1},\dim D_{2j^{\prime}-1},r}$ for $r>0$ with a suitable diffeomorphism and thus we have a smooth map onto a sufficiently small standard closed disc ${D^{\prime}}_{2j^{\prime}-1} \supset D_{2j^{\prime}-1}$ of dimension $\dim D_{2j^{\prime}-1}$ satisfying $S_{a(j^{\prime})} \supset {D^{\prime}}_{2j^{\prime}-1} \supset {\rm Int} {D^{\prime}}_{2j^{\prime}-1} \supset D_{2j^{\prime}-1}$. We can take a sufficiently small standard closed disc ${D^{\prime}}_{2j^{\prime}} \supset D_{2j^{\prime}}$ of dimension $\dim D_{2j^{\prime}}$ satisfying $S_{b(j^{\prime})} \supset {D^{\prime}}_{2j^{\prime}} \supset {\rm Int} {D^{\prime}}_{2j^{\prime}} \supset D_{2j^{\prime}}$ and consider the product map of the previous smooth map and the identity map ${\rm id}_{{D^{\prime}}_{2j^{\prime}}}$. We compose the resulting map with a suitable diffeomorphism respecting FIGURE \ref{fig:2} (or \ref{fig:1}). More precisely, in this context, we construct the map so that the target space is regarded as $c_{a(j^{\prime})}({D^{\prime}}_{2j^{\prime}-1}) \times c_{b(j^{\prime})}({D^{\prime}}_{2j^{\prime}})$ including $p_{j^{\prime}}$ in the interior: we can say that we compose the product map of the immersions $c_{a(j^{\prime})}$ and $c_{b(j^{\prime})}$. 

Last, we change the map into a desired map so that the resulting singular value set is as presented and respects FIGURE \ref{fig:2} (or \ref{fig:1}). We can restrict the map to a total space of a trivial $D^{m-n}$-bundle over the target space, identified with $c_{a(j^{\prime})}({D^{\prime}}_{2j^{\prime}-1}) \times c_{b(j^{\prime})}({D^{\prime}}_{2j^{\prime}})$. For the composition of ${\tilde{f}}_{m-n+\dim D_{2j^{\prime}},\dim D_{2j^{\prime}},r^{\prime}}$ for $r^{\prime}>0$ with a suitable diffeomorphism, we can restrict the map to a total space of a trivial $D^{m-n}$-bundle over the target space, diffeomorphic to a standard closed disc of dimension $\dim D_{2j^{\prime}}$. The total space is regarded as a submanifold of the domain of the original map and we restrict the composition of ${\tilde{f}}_{m-n+\dim D_{2j^{\prime}},\dim D_{2j^{\prime}},r^{\prime}}$ for $r^{\prime}>0$ with the suitable diffeomorphism to the complement of the submanifold in the domain, which is also a compact submanifold of dimension $m-n+\dim D_{2j^{\prime}}$ of the domain. We consider (the composition of) the product map of this restriction and the identity map ${\rm id}_{{D^{\prime}}_{2j^{\prime}-1}}$ (with a suitable diffeomorphism). We replace the original projection of the trivial $D^{m-n}$-bundle over the target space, identified with $c_{a(j^{\prime})}({D^{\prime}}_{2j^{\prime}-1}) \times c_{b(j^{\prime})}({D^{\prime}}_{2j^{\prime}})$, containing $p_{j^{\prime}}$ in the interior, by this, in a suitable way. We have a desired map.

The map is said to be a {\it local canonical fold map around a crossing} for an ATSS operation. 

We can see the following corollary easily and we use this implicitly in various scenes of the present paper. 

\begin{Cor}
\label{cor:1}
Let $f$ be a stable fold map $f:M \rightarrow N$ from an $m$-dimensional closed and connected manifold into an $n$-dimensional manifold with no boundary satisfying $m-n>1$. If an ATSS operation is demonstrateed to $f$ and a new map $f^{\prime}$ is obtained, then $W_f$ is a proper subset of $W_{{f}^{\prime}}$ such that for the map $\bar{{f}^{\prime}}:W_{f^{\prime}} \rightarrow N$, the restriction to $W_f$ is $\bar{f}:W_f \rightarrow N$.
\end{Cor}

\section{Examples, construction of new families of stable fold maps (and their Reeb spaces) and a relation between the cohomology rings of their Reeb spaces and manifolds admitting the maps.}
\label{sec:3}
We will present new examples of stable fold maps by applying ATSS operations to fundamental fold maps such as some special generic maps and investigate homology groups and cohomology rings of the resulting Reeb spaces.

Such studies were also demonstrated in \cite{kitazawa5} and \cite{kitazawa6}. However, obtained examples are simple fold maps. We obtain stable fold maps which may not be simple and calculate the cohomology rings of their Reeb spaces. We compare some of the new results to some of the known results in these articles.

We first obtain examples as Proposition \ref{prop:3}. 
We introduce fundamental terminologies. A homology class of a topological space (regarded as a polyhedron) is said to {\it be represented} by another comoact polyhedron (or as a specific case a manifold regarded as a polyhedron in a natural way) if the class is represented by a cycle obtained in a canonical way from the latter polyhedron. 
Throughout this section, for an integer $i \geq 0$, the {\it $i$-th module} of a graded commutative algebra $A$ over a commutative ring $R$ is the module consisting of all elements of degree $i$ and the $0$-th module is assumed to be $R$ forgetting the ring structure where the action by $R$ is defined in a canonical way. 

\begin{Def}
For two graded commutative algebras $A_1$ and $A_2$ over a commutative ring $R$, a graded commutative algebra $A$ over $R$ is {\it obtained by defining the $0$-th modules in a canonical way from the direct sum} if it is a graded commutative algebra over $R$ satisfying the following properties.
\begin{enumerate}
\item For an integer $i>0$, the $i$-th module is the direct sum of the $i$-th module of $A_1$ and the $i$-th module of $A_2$.
\item For a pair $(a_{i_j,1},a_{i_j,2}) \in A_1 \oplus A_2$ of elements of degree $i_j>0$ for $j=1,2$, they are defined as elements of degree $i_1$ and $i_2$, respectively, and the product is $(a_{i_1,1}a_{i_2,1},a_{i_1,2}a_{i_2,2}) \in A_1 \oplus A_2$ and an element of degree $i_1+i_2$.
\item For $r \in R$, which is an element of degree $0$, and a pair $(a_{i,1},a_{i,2}) \in A_1 \oplus A_2$ of elements of degree $i>0$, they are defined as elements of degree $0$ and degree $i$, respectively, and the product is $(ra_{i,1},ra_{i_2}) \in A_1 \oplus A_2$ and an element of degree $i$.
\end{enumerate}
\end{Def}
\begin{Prop}
\label{prop:3}
Let $R$ be a PID having a unique identity element $1 \in R$ satisfying $1 \neq 0 \in R$. Let $k \in R$ be represented as $k_0 \in \mathbb{Z}$ times the identity element $1$.
Let $f:M \rightarrow N$ be a stable fold map on an $m$-dimensional closed and connected manifold into an $n$-dimensional connected and open manifold satisfying $m-n>1$. We also assume that $n$ is divisible by $4$.
In this situation, by an ATSS operation to $f$, we have a new fold map $f^{\prime}$ satisfying the following properties.
\begin{enumerate}
\item $H_i(W_{f^{\prime}};R)$ is isomorphic to $H_i(W_f;R)$ for $i \neq \frac{n}{2}$ and $H_i(W_f;R) \oplus R \oplus R$ for $i=\frac{n}{2},n$.
\item The cohomology ring $H^{\ast}(W_{f^{\prime}};R)$ is isomorphic to and identified with a graded commutative algebra obtained by defining the $0$-th modules in a canonical way from the direct sum of the cohomology ring $H^{\ast}(W_f;R)$ and a graded commutative algebra $A_R$ over $R$. We denote the $i$-th module by $A_{R,i}$. $A_{R,i}$ is zero
 for $i \neq 0,\frac{n}{2},n$ and isomorphic to $R \oplus R$ for $i =\frac{n}{2},n$. 
 We can find a generator $\{{\nu}_{R,\frac{n}{2},1},{\nu}_{R,\frac{n}{2},2}\}$ of $A_{R,\frac{n}{2}}$ and a generator $\{{\nu}_{R,n,1},{\nu}_{R,n,2}\}$ of $A_{R,n}$
 satisfying the following properties where we consider products as explained in the definition of an algebra obtained by defining the $0$-th modules in a canonical way from a direct sum of two graded commutative algebras. 
 .
\begin{enumerate}
\item The square of ${\nu}_{R,\frac{n}{2},i}$ vanishes {\rm (}$i=1,2${\rm )}.
\item The product ${\nu}_{R,n,1}{\nu}_{R,n,2}$ and $k{\nu}_{R,n,1}+k{\nu}_{R,n,2}$ coincide.
\end{enumerate}
\end{enumerate}
\end{Prop}
\begin{proof}
In the proof, we abuse notation in section \ref{sec:2}: especially notation in Definition \ref{def:3} and around this. Set $l=1$ in Definition \ref{def:3}(--\ref{def:6}). Let us find a suitable normal system of submanifolds compatible with $f$ and denote this by $\{(S_1,N(S_1),c_1:N(S_1) \rightarrow P),(S_2,N(S_2),c_2:N(S_2) \rightarrow P)\}$. We take this so that $S_1$ and $S_2$ are standard spheres of dimension $\frac{n}{2}$.
We take an open disc $U$ in the regular value set sufficiently close to the intersection of the image of a connected component of the singular set consisting of singular points of index $0$ and the complement of the corner in the boundary of the image of the map: the image $f(M)$ is regarded as a manifold which may have corners and smoothly embedded in the target manifold. Singular points of index $0$ exist and we can take $U$ as this. In fact, $M$ is closed, $N$ is not compact (closed) and thus there exists a singular point of index $0$. 

Let us assume $k_0 \geq 0$.

We can take $\{(S_1,N(S_1),c_1:N(S_1) \rightarrow P),(S_2,N(S_2),c_1:N(S_2) \rightarrow P)\}$ so that the family $\{c_1,c_2\}$ has exactly $k_0 \geq 0$ pairs of crossings ($2k_0$ crossings), that the normal bundles of the immersions are trivial and that the relation $c_1(N(S_1)) \bigcup c_2(N(S_2))  \subset U$ holds.
We can demonstrate an ATSS operation whose generating normal system is $c_1(N(S_1)) \bigcup c_2(N(S_2))$ to obtain a new fold map $f^{\prime}$. We use a local canonical fold map around a crossing in section \ref{sec:2} around each crossing of the family $\{c_1(S_1),c_2(S_2)\}$: around the remaining singular values and regular values, we construct product maps of Morse functions with exactly one singular point, which is of index $1$, and identity maps on ($n-1$)-dimensional manifolds and trivial $S^{m-n}$-bundles over $n$-dimensional manifolds and last glue all the local maps together. We can demonstrate this construction thanks to the assumption that $U$ is an open ball in the regular value set sufficiently close to the image of a connected component of the singular set consisting of singular points of index $0$.

By the definition of a normal bubbling operation, $W_{f^{\prime}}$ is regarded as a space obtained by attaching a polyhedron $A$ we can obtain by identifying exactly $k_0$-pairs of disjointly embedded PL discs of dimension $n$ of the disjoint union of smooth (PL) manifolds $S_1 \times S^{\frac{n}{2}}$ and $S_2 \times S^{\frac{n}{2}}$ represented by products of two PL discs of dimension $\frac{n}{2}$ to $B:={\bar{f}}^{-1}(c_1(N(S_1)) \bigcup c_2(N(S_2))) \subset W_f$. More precisely, the union of $S_1 \times D^{\frac{n}{2}} \subset S_1 \times S^{\frac{n}{2}}$ and $S_2 \times D^{\frac{n}{2}} \subset S_2 \times S^{\frac{n}{2}}$ is attached
to where $D^{\frac{n}{2}}$ denotes a hemisphere of dimension $\frac{n}{2}$.
Let $i_{B,f}$ and $i_{B,A}$ denote the canonical inclusions of $B$ into $W_f$ and $A$, respectively. We have the following exact sequence.

$\begin{CD}
@>   >> H_i(B;R) @> {i_{B,f}}_{\ast} \oplus {i_{B,A}}_{\ast} >> H_i(W_f;R) \oplus H_i(A;R)
\end{CD}$

$\begin{CD}
@>   >>  H_i(W_{f^{\prime}};R) @>    >> H_{i-1}(B;R) 
\end{CD}$

$\begin{CD}
@>  {i_{B,f}}_{\ast} \oplus {i_{B,A}}_{\ast}  >> H_{i-1}(W_f;R) \oplus H_{i-1}(A;R) @>   >>
\end{CD}$.

By observing the definition of a normal bubbling operation and the topologies of these spaces, the image of ${i_{B,f}}_{\ast}$ is zero except for the case $i=0$ and ${i_{B,A}}_{\ast}$ is injective. 
$H_i(W_{f^{\prime}};R)$ is isomorphic to $H_i(W_f;R)$ for $i \neq \frac{n}{2}$ and $H_i(W_f;R) \oplus R \oplus R$ for $i=\frac{n}{2},n$: we set suitable identifications. In the case $i= \frac{n}{2}$, the summands $R$ are seen to be generated by the class represented by $\{{\ast}_j\} \times S^{\frac{n}{2}}$ in the original manifolds $S_j \times S^{\frac{n}{2}}$ to obtain $A$ where ${\ast}_j$ is a suitable point in $S_j$ for $j=1,2$. In the case $i=n$, the summands $R$ are seen to be generated by the classes represented by two $n$-dimensional polyhedra in $A$. Remember that $A$ is an $n$-dimensional polyhedron obtained from the original $n$-dimensional closed, connected and orientable manifolds $S_j \times S^{\frac{n}{2}}$.

We discuss the cohomology rings. By the construction, the resulting cohomology ring is represented as an algebra obtained by defining the $0$-th modules in a canonical way from the direct sum of that of $W_f$ and a new graded commutative algebra $A_R$: we denote the $i$-th module of $A_R$ by $A_{R,i}$. This is zero except $i=\frac{n}{2},n$. In these two cases $i=\frac{n}{2},n$, the modules are isomorphic to $R \oplus R$. We investigate the product of two elements of degree $\frac{n}{2}$. We can see that the classes are regarded as classes represented by the original manifolds $S_j \times S^{\frac{n}{2}}$ (considering the situation they are deformed and attached to $W_f$). 

We discuss the case where the number $k_0 \geq 0$ of the pairs of crossings is $1$. In this case, the square of each element in a suitable generator $\{{\nu}_{R,\frac{n}{2},1},{\nu}_{R,\frac{n}{2},2}\}$ of the module $A_{R,\frac{n}{2}}$ vanishes and ${\nu}_{R,\frac{n}{2},1}{\nu}_{R,\frac{n}{2},2}$ can be an element of a suitable generator $\{{\nu}_{R,n,1},{\nu}_{R,n,2}\}$ of $A_{R,n}$ and also $0$ and we explain the reason. 

First we explain the classes in ${\nu}_{R,\frac{n}{2},1}{\nu}_{R,\frac{n}{2},2}$ and $\{{\nu}_{R,n,1},{\nu}_{R,n,2}\}$. We can define a cohomology class regarded as the dual of the homology class represented by $\{{\ast}_1\} \times S^{\frac{n}{2}}$ ($\{{\ast}_2\} \times S^{\frac{n}{2}}$) before: the value at this class is the identity element $1 \in R$ and the values at classes in $H_{\frac{n}{2}}(W_f;R) \oplus \{0\} \subset H_{\frac{n}{2}}(W_f;R) \oplus A_R$ and at the class represented by $\{{\ast}_2\} \times S^{\frac{n}{2}}$ (resp. $\{{\ast}_1\} \times S^{\frac{n}{2}}$) are zero. We define ${\nu}_{R,\frac{n}{2},1}$ and ${\nu}_{R,\frac{n}{2},2}$ as these classes.

We evaluate the value of the product ${\nu}_{R,\frac{n}{2},1}{\nu}_{R,\frac{n}{2},2}$ at the classes represented by $n$-dimensional polyhedra in $A$ and generating the summands $R$. More precisely, the classes are obtained after the original manifolds $S_j \times S^{\frac{n}{2}}$ are deformed and attached to $W_f$. Similarly, we define ${\nu}_{R,n,1}$ and ${\nu}_{R,n,2}$ as the duals of these classes. 

For the original manifold $S_2 \times S^{\frac{n}{2}}$, the class represented by $S_2 \times \{{\ast}^{\prime}\} \subset S_2 \times S^{\frac{n}{2}}$ can be mapped to the class represented by $\{{\ast}_1\} \times S^{\frac{n}{2}} \subset S_1 \times S^{\frac{n}{2}}$ in the original manifold deformed and attached to $W_f$ and regarded as a subspace in $W_{f^{\prime}}$ where ${\ast}^{\prime}$ is a point. Moreover, if we treat this sphere as an embedding, then via a suitable homotopy, we can change this to a PL homeomorphism onto a sphere regarded as the preimage ${\bar{f}}^{-1}(c_1(F))$ of the image $c_1(F)$ of a fiber $F$ of the bundle $N(S_1)$ by the map $\bar{f}:W_f \rightarrow N$. The class represented by $S_2 \times \{{\ast}^{\prime}\} \subset S_2 \times S^{\frac{n}{2}}$ can be also mapped to the zero class if we demonstrate a normal bubbling operation in another suitable way. 

We explain on the class represented by $S_2 \times \{{\ast}^{\prime}\} \subset S_2 \times S^{\frac{n}{2}}$ in the resulting Reeb space $W_{f^{\prime}}$. We consider one point in the pair $(p_1,p_2)$ of the crossings of the family $\{{c_1} {\mid}_{S_1},{c_2} {\mid}_{S_2}\}$ of the immersions: take $p_2$. We can take $D_{1}$, $D_{2}$, $D_{3}$ and $D_{4}$ as in Definition \ref{def:3}. The key ingredient is observing the topologies around $p_2$ and ${\bar{f}}^{-1}(p_2)$. We can consider another normal bubbling operation to $f$ to exchange the types of the topology around $p_2$ and ${\bar{f}}^{-1}(p_2)$ without changing other parts including the structure around $p_1$ and ${\bar{f}}^{-1}(p_1)$ from the original $f^{\prime}$. We explain this.
This figure represents the topology of $W_{f^{\prime}}$ seen from directly above and a case for the $2$-dimensional case and this also accounts for a case for general $n$. $D_j \times \{{\ast}^{\prime \prime}\} \subset S_1 \times S^{\frac{n}{2}}$ ($j=1,3$) and $D_j \times \{{\ast}^{\prime}\} \subset S_2 \times S^{\frac{n}{2}}$ ($j=2.4$) denote the subpolyhedra (discs) attached to the original Reeb space $W_f$ where ${\ast}^{\prime \prime}$ is a suitable point. Dots denote $S_1 \times \{{\ast}^{\prime \prime}\} \subset S_1 \times S^{\frac{n}{2}}$ and $S_2 \times \{{\ast}^{\prime}\} \subset S_2 \times S^{\frac{n}{2}}$, after they are deformed and attached to the Reeb space. We represent the original manifold $S_1 \times S^{\frac{n}{2}}$ by $S_1 \times ({D^{\frac{n}{2}}}_1 \bigcup  {D^{\frac{n}{2}}}_2)$ where ${D^{\frac{n}{2}}}_1$ and ${D^{\frac{n}{2}}}_2$ are hemispheres whose boundaries are an equator of $S^{\frac{n}{2}}$. We can demonstrate an ATSS operation so that $D_2 \times {\{\ast\}}^{\prime \prime}$ and $\{q_1\} \times {D^{\frac{n}{2}}}_1 \subset D_1 \times S^{\frac{n}{2}}$ agree in $W_{f^{\prime}}$ for a suitable point $q_1$. We can demonstrate an ATSS operation so that $D_4 \times \{{\ast}^{\prime \prime}\}$ and $\{q_2\} \times {D^{\frac{n}{2}}}_1 \subset D_3 \times S^{\frac{n}{2}}$ agree and also demonstrate an ATSS operation so that $D_4 \times \{{\ast}^{\prime \prime}\}$ and $\{q_2\} \times {D^{\frac{n}{2}}}_2 \subset D_3 \times S^{\frac{n}{2}}$ agree in $W_{f^{\prime}}$ for a suitable point $q_2$.

\begin{figure}
\includegraphics[width=30mm]{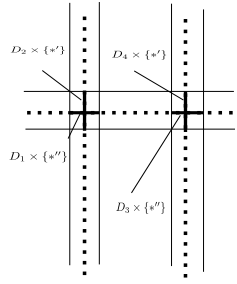}
\caption{The topology of $W_{f^{\prime}}$ seen from directly above. $D_{j_1} \times \{{\ast}^{\prime}\},D_{j_1} \times \{{\ast}^{\prime \prime}\} \subset S_{j_2} \times S^{\frac{n}{2}}$ denote the subpolyhedra (discs), after they are deformed and attached to the riginal Reeb space $W_f$, and dots are for $S_1 \times \{{\ast}^{\prime \prime}\} \subset S_1 \times S^{\frac{n}{2}}$ and $S_2 \times \{{\ast}^{\prime}\} \subset S_2 \times S^{\frac{n}{2}}$, after they are deformed and attached to the Reeb space.}
\label{fig:3}

\end{figure}
This yields the fact that the value of the product ${\nu}_{R,\frac{n}{2},1}{\nu}_{R,\frac{n}{2},2}$ before at the class represented by the two $n$-dimensional polyhedra in $A$ or the classes obtained after the original manifolds $S_j \times S^{\frac{n}{2}}$ are deformed and attached to $W_f$ can be $k{\nu}_{R,n,1}+k{\nu}_{R,n,2}$. The product can be also zero if we demonstrate the operation in a suitable way.

For a general $k_0 \geq 0$ and also for $k_0 <0$, we can argue similarly. For the case where the number $k_0$ is $0$, see also the proofs of some propositions and theorems of \cite{kitazawa6}. Note also that we can demonstrate the procedure such that the number of the pairs of the crossings is larger than $k_0 \geq 0$ ($|k_0| \geq 0$) and that in this case we can argue similarly.

This completes the proof.
\end{proof}

The following is obtained by applying Proposition \ref{prop:3} inductively and rigorous proofs are left to readers.

\begin{Thm}
\label{thm:1}
Let $R$ be a PID having a unique identity element $1 \in R$ satisfying $1 \neq 0 \in R$. Let $k_1>0$ and $k_2 \geq 0$ be integers and $\{k_{1,j}\}_{j=1}^{k_1} \subset R$ be a sequence of elements of $R$ such that $k_{1,j}$ is represented as $k_{0,1,j} \in \mathbb{Z}$ times the identity element $1$.
Let $f:M \rightarrow N$ be a stable fold map on an $m$-dimensional closed and connected manifold $M$ into an $n$-dimensional connected and open manifold $N$ satisfying $m-n>1$. We also assume that $n$ is divisible by $4$.

In this situation, by an ATSS operations to $f$, we have a new fold map $f^{\prime}$ satisfying the following properties.
\begin{enumerate}
\item $H_i(W_{f^{\prime}};R)$ is isomorphic to $H_i(W_f;R)$ for $i \neq \frac{n}{2}$ and $H_i(W_f;R) \oplus R^{2k_1+k_2}$ for $i=\frac{n}{2},n$.
\item The cohomology ring $H^{\ast}(W_{f^{\prime}};R)$ is isomorphic to and identified with an algebra obtained by defining the $0$-th modules in a canonical way from the direct sum of the cohomology ring $H^{\ast}(W_f;R)$ and a graded commutative algebra $A_R$ over $R$. We denote the $i$-th module of $A_R$ by $A_{R,i}$. $A_{R,i}$ is zero
 for $i \neq 0,\frac{n}{2},n$ and isomorphic to $R$ for $i=0$ and $R^{2k_1+k_2}$ for $i =\frac{n}{2},n$. If we take a suitable generator $\{a_j\}_{j=1}^{2k_1+k_2}$ of $\{0\} \oplus A_{R,\frac{n}{2}} \subset H^{\frac{n}{2}}(W_f;R) \oplus A_{R,\frac{n}{2}}$ and a suitable generator $\{b_j\}_{j=1}^{2k_1+k_2}$ of $\{0\} \oplus A_{R,n} \subset H^{n}(W_f;R) \oplus A_{R,n}$, then for products
  as explained in the definition of an algebra obtained by defining the $0$-th modules in a canonical way from a direct sum of two graded commutative algebras the following properties hold. 
\begin{enumerate}
\item The product of $a_{2j-1}$ and $a_{2j}$ is $k_{1,j}b_{2j-1}+k_{1,j}b_{2j}$ for any integer satisfying $1 \leq j \leq k_1$.
\item The product of $a_{j_1}$ and $a_{j_2}$ vanishes for any pair of integers satisfying $1 \leq j_1,j_2 \leq 2k_1$ except cases $(j_1,j_2)=(2j-1,2j)$ for any integer satisfying $1 \leq j \leq k_1$.
\end{enumerate}
\end{enumerate}
\end{Thm}

As a more precise explanation, we demonstrate a procedure for $k_1$ times in Proposition \ref{prop:3} and demonstrate an ATSS such that the generating normal system is a disjoint union of $k_2$ embedded standard  spheres of dimension $\frac{n}{2}$ (the immersion is an embedding and we can demonstrate this procedure easily). We  present another theorem. Before this, we introduce the following proposition, which is as a fundamental proposition shown in \cite{kitazawa5} and \cite{kitazawa6}. We can show this by applying a method as in the proof of Proposition \ref{prop:3} which is simpler. We apply this in the last step of the proof of Theorem \ref{thm:2} later and through the argument we can also know an important ingredient of the proof of this proposition.

\begin{Prop}
\label{prop:4}
Let $R$ be a PID having a unique identity element $1 \in R$ satisfying $1 \neq 0 \in R$. Let $k \in R$ be represented as $k_0 \in \mathbb{Z}$ times the identity element $1$.
Let $f:M \rightarrow N$ be a stable fold map on an $m$-dimensional closed and connected manifold into an $n$-dimensional manifold with no boundary satisfying $m-n>1$. Let $N$ be not closed. We also assume that $n$ is divisible by $4$.
In this situation, by an ATSS operation to $f$, we have a new fold map $f^{\prime}$ satisfying the following properties.
\begin{enumerate}
\item $H_i(W_{f^{\prime}};R)$ is isomorphic to $H_i(W_f;R)$ for $i \neq \frac{n}{2}$ and $H_i(W_f;R) \oplus R$ for $i=\frac{n}{2},n$.
\item The cohomology ring $H^{\ast}(W_{f^{\prime}};R)$ is isomorphic to and identified with a graded commutative algebra obtained by defining the $0$-th modules in a canonical way from the direct sum of the cohomology ring $H^{\ast}(W_f;R)$ and a graded commutative algebra $A_R$ over $R$. We denote the $i$-th module by $A_{R,i}$. $A_{R,i}$ is zero for $i \neq 0,\frac{n}{2},n$ and isomorphic to $R$ for $i =\frac{n}{2},n$. 
\end{enumerate}
\end{Prop}  
\begin{Thm}
\label{thm:2}
Let $R$ be a PID having a unique identity element $1 \in R$ satisfying $1 \neq 0 \in R$. Let $k_1,k_2>0$ and $k_3 \geq 0$ be integers and $\{k_{1,j}\}_{j=1}^{2k_1k_2} \subset R$, $\{k_{2,j}\}_{j=1}^{k_2} \subset R$ and $\{k_{3,j}\}_{j=1}^{k_1k_3} \subset R$ be sequences of elements of $R$ such that $k_{i,j}$ is represented as $k_{0,i,j} \in \mathbb{Z}$ times the identity element $1$.
Let $f:M \rightarrow N$ be a stable fold map on an $m$-dimensional closed and connected manifold into an $n$-dimensional connected  and open manifold satisfying $m-n>1$. We also assume that $n$ is divisible by $4$.
In this situation, we can change $f$ in the following way into a stable fold map $f_0$ into $N$.
First, we remove the interior of manifold $D$ represented as a disjoint union of $k_1$ copies of a product of a standard closed disc of dimension $\frac{n}{2}+1$ and ($\frac{n}{2}-1$)-dimensional standard sphere smoothly embedded in the regular value set sufficiently close to the intersection of the image of a connected component of the singular set consisting of singular points of index $0$ and the complement of the corner in the boundary of the image of the map: the image $f(M)$ is regarded as a manifold which may have corners and smoothly embedded in the target manifold. After that we attach a product map of a height function on a unit disc of dimension $m-n+1$ and the identity map on $\partial D$ instead by diffeomorphisms between the boundaries so that the resulting singular value set $f_0(S(f_0))$ is the disjoint union of the original one and an embedded manifold diffeomorphic to $\partial D$.
Moreover, by an ATSS operations to $f_0$, we have a new fold map $f^{\prime}$ satisfying the following properties.
\begin{enumerate}
\item $H_i(W_{f^{\prime}};R)$ is isomorphic to $H_i(W_f;R)$ for $i \neq \frac{n}{2}$, $H_i(W_f;R) \oplus R^{k_1} \oplus R^{2k_2} \oplus R^{k_3}$ for $i=\frac{n}{2}$ and $H_n(W_f;R) \oplus R^{2k_2} \oplus R^{k_3}$ for $i=n$.
\item The cohomology ring $H^{\ast}(W_{f^{\prime}};R)$ is isomorphic to and identified with an algebra obtained by defining the $0$-th modules in a canonical way from the direct sum of the cohomology ring $H^{\ast}(W_f;R)$ and a graded commutative algebra $A_R$ over $R$. We denote the $i$-th module by $A_{R,i}$ and this is zero for $i \neq 0,\frac{n}{2},n$, isomorphic to $R$ for $i=0$, isomorphic to $R^{k_1} \oplus R^{2k_2} \oplus R^{k_3}$ for $i =\frac{n}{2}$ and isomorphic to $R^{2k_2} \oplus R^{k_3}$ for $i=n$.
 If we take a suitable generator $\{a_{1,j}\}_{j=1}^{k_1} \sqcup \{a_{2,j}\}_{j=1}^{2k_2} \sqcup \{a_{3,j}\}_{j=1}^{k_3}$ of $\{0\} \oplus A_{R,\frac{n}{2}} \subset H^{\frac{n}{2}}(W_f;R) \oplus A_{R,\frac{n}{2}}$ and a suitable
  generator $\{b_j\}_{j=1}^{2k_2} \sqcup \{c_j\}_{j=1}^{k_3}$ of $\{0\} \oplus A_{R,n} \subset H^{n}(W_f;R) \oplus A_{R,n}$, then for products of the form $a_{j_1,{j_1}^{\prime}}a_{j_2{,j_2}^{\prime}}$ the following properties hold where products are as explained in the definition of an algebra obtained by defining the $0$-th modules in a canonical way from a direct sum of two graded commutative algebras.
\begin{enumerate}
\item The product of $a_{2,2j-1}$ and $a_{2,2j}$ is $k_{2,j}b_{2j-1}+k_{2,j}b_{2j}$ for any integer satisfying $1 \leq j \leq k_2$.
\item The product of $a_{1,j_1}$ and $a_{2,j_2}$ is $k_{1,2(j_1-1)k_2+j_2}b_{j_2}$ for any pair of integers satisfying $1 \leq j_1 \leq k_1$ and $1 \leq j_2 \leq 2k_2$.
\item The product of $a_{1,j_1}$ and $a_{3,j_2}$ is $k_{3,(j_1-1)k_3+j_2}c_{j_2}$ for any pair of integers satisfying $1 \leq j_1 \leq k_1$ and $1 \leq j_2 \leq k_3$.
\item Products which do not have one of the forms presented above vanish. 
\end{enumerate}
\end{enumerate}
\end{Thm}

We present a proof for the case $k_1=1$. More rigorous explanations are needed on the cohomology rings in the proof and they are explained in the proofs of propositions and theorems in \cite{kitazawa6}. 
Moreover, we can show this for a general $k_1$ and these proofs are left to readers.
\begin{proof}[A proof for the case $k_1=1$.]
We can see that $H_i(W_{f_0};R)$ is isomorphic to $H_i(W_f;R)$ for $i \neq \frac{n}{2}$, $H_i(W_f;R) \oplus R^{2k_2}$ for $i=\frac{n}{2}$ and $H_n(W_f;R)$ for $i=n$. The cohomology ring is isomorphic to an algebra obtained by identifying the $0$-th modules in a canonical way in the direct sum of $H^{\ast}(W_f;R)$ and an algebra $A_{0,R}$: we denote the $i$-th module by $A_{0,R,i}$ and this is zero
 for $i \neq 0,\frac{n}{2}$ and isomorphic to $R$ for $i=0,\frac{n}{2}$. We explain about an ATSS operation to $f_0$.

We can choose $2k_2$ immersions of $S^{\frac{n}{2}}$ into the regular value set sufficiently close to the new connected component of the singular value set $f_0(S(f_0))$. We can take the immersions so that for any integer $1 \leq j \leq k_2$ the ($2j-1$)-th and $2j$-th immersions satisfy the following properties and that for any pair $(j_1,j_2)$ of integers satisfying $1 \leq j_1<j_2 \leq 2k_2$ and $(j_1,j_2) \neq (2j-1,2j)$ for any integer $1 \leq j \leq k_2$ the images of the $j_1$-th and $j_2$-th immersions are disjoint.
\begin{enumerate}
\item At fundamental classes of the domain of the immersions into the (interior of the) image $f_0(M_0)$, which is regarded as a manifold which may have corners and smoothly embedded in the target manifold, the values of the homomorphisms induced by the inclusions are $k_{1.2j-1}={\Sigma}_{j^{\prime}=(2j-2)k_1+1}^{(2j-1)k_1} k_{1,j^{\prime}}$ times a generator and $k_{1,2j}={\Sigma}_{j^{\prime}=(2j-1)k_1+1}^{2jk_1} k_{1,j^{\prime}}$ times a generator of the class represented by a sphere parallel to an $\frac{n}{2}$-dimensional sphere in the new connected component in the resulting singular value set $f_0(S(f_0))$, respectively: we present the representations $k_{1.2j-1}={\Sigma}_{j^{\prime}=(2j-2)k_1+1}^{(2j-1)k_1} k_{1,j^{\prime}}$ and $k_{1,2j}={\Sigma}_{j^{\prime}=(2j-1)k_1+1}^{2jk_1} k_{1,j^{\prime}}$ considering the case for general $k_1$.
The generator is also a generator of the $\frac{n}{2}$-th homology group of the connected component, diffeomorphic to $S^{\frac{n}{2}-1} \times S^{\frac{n}{2}}$. This argument is based on fundamental arguments on differential topological theory of immersions and embeddings. For more rigorous discussions, see also several proofs of propositions and theorems of \cite{kitazawa6}.  
\item The number of the crossings of the family of these two immersions is $2|k_{2,j}|$ for $n>2$ and larger than or equal to this for $n=2$: this is based on fundamental properties of immersions and embeddings of curves into planes and surfaces. 
\end{enumerate}
 By applying (essential arguments in the proof of) Proposition \ref{prop:3} one after another observing the changes of the topologies of the Reeb spaces and so on, in the case $k_3=0$, we have a result. 
 For a general $k_3$, we need to demonstrate the following procedure in addition and we apply Proposition \ref{prop:4} with several arguments in (the proofs of) propositions and theorems in \cite{kitazawa6}.
\begin{enumerate}
\item We take $k_3$ embeddings of $S^{\frac{n}{2}}$ into the regular value set sufficiently close to the new connected component of the singular value set $f_0(S(f_0))$ where we consider the map before demonstrating the previous operation. We can take the embeddings so that the images are disjoint and that for any integer $1 \leq j \leq k_3$ at a fundamental class of the domain of the embedding into the (interior of the) image $f_0(M_0)$, which is regarded as a manifold which may have corners and smoothly embedded in the target manifold, the value of the homomorphism induced by the embedding is $k_{3.j}={\Sigma}_{j^{\prime}=(j-1)k_1+1}^{j k_1} k_{3,j^{\prime}}$ times the class represented by a sphere parallel to an $\frac{n}{2}$-dimensional sphere in the new connected component in the resulting singular value set $f_0(S(f_0))$: we present the representation $k_{3.j}={\Sigma}_{j^{\prime}=(j-1)k_1+1}^{j k_1} k_{3,j^{\prime}}$ considering the case for general $k_1$.
The generator is also a generator of the $\frac{n}{2}$-th homology group of the connected component, diffeomorphic to $S^{\frac{n}{2}-1} \times S^{\frac{n}{2}}$.  
\item We demonstrate an ATSS whose generating normal system is the union of the images of the embeddings just before.
\end{enumerate}
This completes the proof.

\end{proof}

\begin{Ex}
In the situation of Theorem \ref{thm:2}, we consider the case where $k_1+2k_2+k_3=3$. Consider the case where $k_2=0$ or $k_{2,j}=0$ for all $j$. In this case, we can easily see that there exists a submodule of rank $2$ of the $\frac{n}{2}$-th module of the PID $A_R$ consisting of elements such that the squares vanish.
Consider the case where $R=\mathbb{Z}, \mathbb{Q}$, for example, $k_1=k_2=1$, $k_{0,1,1}=1$, $k_{0,1,2}=k_{0,1,2k_1k_2}=2$ and $k_{0,2,1}=k_{0,2,k_2}=3$.
We consider the square of an element $r_1a_{1,1}+r_2a_{2,1}+r_3a_{2,2}$ where $r_1,r_2,r_3 \in R$. The square is $2r_1r_2k_{1,1}b_1+2r_1r_3k_{1,2}b_2+2r_2r_3k_{2,1}b_1+2r_2r_3k_{2,1}b_2=(2r_1r_2+6r_2r_3)b_1+(4r_1r_3+6r_2r_3)b_2$. 
We consider the case where this is zero. Suppose that $r_2$ is zero. In this case, $r_1$ or $r_3$ must be zero.  Suppose that $r_2$ is not zero. In this case, $r_1=(-3)r_3$ and as a result $r_3$ is zero or represented as $r_2=2r_3$.  
We have that there exists no submodule of rank $2$ of the $\frac{n}{2}$-th module of $A_R$ consisting of elements such that the squares vanish. Such a case does not appear in \cite{kitazawa6} while the former case is discussed there. Such cases happen generally of course.
\end{Ex}

Last we extend Proposition \ref{prop:2}
\begin{Prop}
\label{prop:5}
Proposition \ref{prop:2} holds even if we weaken the condition so that the restrictions to the singular sets may have crossings which are normal and each connected component of the preimage of each of which contains at most two singular points. 
\end{Prop}
The {\it piecewise smooth category} is the category whose objects are smooth manifolds with canonically defined PL structures and whose morphisms are piecewise smooth maps between the manifolds. The category is known to be equivalent to the PL category.
\begin{proof}[The sketch of a proof]
This is essentially owing to the discussions in referred articles and \cite{saeki3}. More rigorous proofs are left to readers.

We will construct a manifold bounded by the original manifold collapsing to the Reeb space in the piecewise smooth category.

Note that for each point in the image of the map $f$ and an $n$-dimensional small standard closed disc containing this in the interior, each connected component of the preimage is one of the following types as smooth manifolds (which may have corners),
\begin{enumerate}
\item A product of an ($m-n$)-dimensional standard sphere and an $n$-dimensional standard closed disc. 
\item A product of an manifold obtained by removing the interior of the union of three disjointly and smoothly embedded standard closed discs in $S^{m-n+1}$ and an ($n-1$)-dimensional standard closed disc.   
\item A product of an manifold obtained by removing the interior of the union of four disjointly and smoothly embedded standard closed discs in $S^{m-n+1}$, a closed interval $I$ and an ($n-2$)-dimensional standard closed disc. 
\end{enumerate}

There exist three types of the topologies of small regular neighborhoods of points in the Reeb space. Each type of these types of the topologies is for each case above. All points in the Reeb space of each type form manifolds whose dimensions are $n$, $n-1$, and $n-2$, respectively. 

For each case, we can construct bundles whose fibers are as above in the piecewise linear category and we can construct bundles whose fibers are $D^{m-n+1}$, $D^{m-n+2}$, or $D^{m-n+2} \times I$ in the category whose subbundles obtained by restricting the fibers to suitable compact submanifolds of the boundaries of the discs $D^{m-n+1}$ or $D^{m-n+2}$ are the original bundles: note that the dimensions of the suitable compact submanifolds are same as those of the boundaries and that in the last case we restrict $D^{m-n+2} \times I$ to $\partial D^{m-n+2} \times I =S^{m-n+1} \times I$ first. We can locally construct these bundles and we can glue them in the piecewise smooth category (PL category).

This yields a desired ($m+1$)-dimensional compact PL manifold collapsing to $W_f$, which is an $n$-dimensional polyhedron.
The resulting ($m+1$)-dimensional manifold is regarded as a manifold obtained by attaching handles whose indices are larger than or equal to $m-n$ to $M \times \{1\} \subset M \times [0,1]$ where we discuss in the PL category.
\end{proof}
We can apply this to some explicit fold maps obtained in Theorems \ref{thm:1} and \ref{thm:2} and so on.

\section{Acknowledgement.}
The author is a member of JSPS KAKENHI Grant Number JP17H06128 "Innovative research of geometric topology and singularities of differentiable mappings"
(https://kaken.nii.ac.jp/en/grant/KAKENHI-PROJECT-17H06128/). This work is supported by this project.

\end{document}